\documentclass{article}%
\usepackage{amsmath}
\usepackage{amsfonts}
\usepackage{amssymb}
\usepackage{graphicx}%
\providecommand{\U}[1]{\protect\rule{.1in}{.1in}}
\newtheorem{theorem}{Theorem}

\newtheorem{proposition}[theorem]{Proposition}

\newenvironment{proof}[1][Proof]{\noindent\textbf{#1.} }{\ \rule{0.5em}{0.5em}}
\begin{document}

\title{Fast algorithms for interpolation with clamped $L$-splines of order four}
\author{O. Kounchev, H. Render, G. Simeonov, Ts. Tsachev}
\date{}
\maketitle

\begin{abstract}
	Interpolation and smoothing using cubic and generalized splines are fundamental tools in data analysis and statistical modeling. Recently, fast computational algorithms were developed for natural $L$-splines of order four, which arise as piecewise solutions to the differential operator $L_{\xi}^2 = (\frac{d^2}{dt^2} - \xi^2)^2$. In this paper, we extend this mathematical framework to the important case of clamped (or complete) boundary conditions, where the first derivatives at the interval endpoints are prescribed. We explicitly construct the governing linear system for the interpolation problem and mathematically prove that the resulting tridiagonal matrix is strictly row diagonally dominant, thereby guaranteeing its invertibility and the numerical stability of the fast algorithm. The proposed method is implemented in MATLAB. Furthermore, the developed clamped $L$-splines provide a foundation for constructing multivariate clamped polysplines, which serve as a promising alternative to Physics-Informed Neural Networks (PINNs) for solving partial differential equations in Mathematical Physics. 
\end{abstract}


\section{Introduction}

Interpolation with cubic polynomial splines is a popular method for the
analysis and smoothing of big data -- often used in statistical modelling,
see \cite{deBoor}, \cite{StoerBulirsch} and \cite{LycheSchumaker} for
mathematical aspects, and \cite{GreenSilverman}, \cite{Gu},
\cite{HastieTibshiraniFriedman}, \cite{RamsaySilverman} and \cite{Wahba1990}
for applications. Recently the authors proposed and studied in
\cite{kounchevRenderTsachev-BITpaper} fast algorithms for interpolation and
smoothing for a special class of generalized splines arising as piece-wise
solutions of differential operators, called $L-$splines. In
\cite{kounchevRenderTsachev-BITpaper} it is assumed that the differential
operator $L$ is of the form $L_{\xi}^{2}:=L_{\xi}\circ L_{\xi}$ where
\[
L_{\xi}:=\frac{d^{2}}{dt^{2}}-\xi^{2}\;\;\text{for }\;\;\;\xi\in\mathbb{R}.
\]
Let us recall the definition of an $L_{\xi}^{2}-$spline: assume that $\xi
\geq0$ is a real number and let $t_{1}<t_{2}<\dots<t_{n-1}<t_{n}$ be real
numbers. A function $g(\cdot)\in C^{2}[t_{1},t_{n}]$ is called an $L-$spline
if
\[
L_{\xi}^{2}g(t)=0\;\;\;\text{for }\;\;\;t\in(t_{j},t_{j+1}).
\]
A function $g(\cdot)\in C^{2}[t_{1},t_{n}]$ is called an $L_{\xi}^{2}-$spline
interpolating the data $\{g_{1},g_{2},\dots,g_{n}\}$ at $\{t_{1},t_{2}%
,\dots,t_{n}\}$ if
\[
g(t_{j})=g_{j}\text{ for }j=1,2,\dots,n.
\]
In order to obtain uniqueness for interpolation one has to add boundary
conditions at the endpoints $t_{1}$ and $t_{n}.$ In
\cite{kounchevRenderTsachev-BITpaper} we considered \emph{natural }%
$L-$\emph{splines} requiring the boundary condition
\[
L_{\xi}g(t_{1})=L_{\xi}g(t_{n})=0.
\]

The main aim of the present paper is to show that one can derive fast
algorithms as in \cite{kounchevRenderTsachev-BITpaper} for the case of
\emph{clamped splines} where the values $g^{\prime}(t_{1})$ and $g^{\prime
}(t_{n})$ are prescribed. These splines are also called "complete splines",
see \cite{deBoor} and/or in Matlab. For related work on more general
$L$-splines we refer to \cite{McCa90}, \cite{McCa91} and \cite{KRT2023}. 

\section{Fast interpolation for natural $L_{\xi}^{2}-$splines - a recap}

Suppose that $g\in C^{2}\left[  t_{1},t_{n}\right]  $ is an $L_{\xi}^{2}%
$-spline. In order to compute the $L_{\xi}^{2}$-spline $g$ explicitly one has
to solve for each interval $\left[  t_{j},t_{j+1}\right]  $ with
$j=1,....,n-1,$ the boundary value problem $L_{\xi}^{2}w=0$ with
\[
w\left(  t_{j}\right)  =g_{j},\ w\left(  t_{j+1}\right)  =g_{j+1},\text{ and
}L_{\xi}w\left(  t_{j}\right)  =L_{\xi}g\left(  t_{j}\right)  ,\ L_{\xi
}w\left(  t_{j+1}\right)  =L_{\xi}g\left(  t_{j+1}\right)  .
\]
Thus, in order to compute the natural $L_{\xi}^{2}$-spline $g$ interpolating
the data $z_{j}=g\left(  t_{j}\right)  $ for $j=1,...,n$, we need the
knowledge of the two vectors
\begin{align}
\mathbf{g}^{T}  &  :=\left(  g\left(  t_{1}\right)  ,g\left(  t_{2}\right)
,...,g\left(  t_{n}\right)  \right) \label{g}\\
\widetilde{\gamma}^{T}  &  :=\left(  L_{\xi}g\left(  t_{1}\right)
,...,L_{\xi}g\left(  t_{n}\right)  \right)  . \label{gamma}%
\end{align}
In case of a natural interpolating $L_{\xi}^{2}-$spline we know $\mathbf{g}%
^{T}$ and $L_{\xi}g\left(  t_{1}\right)  =L_{\xi}g\left(  t_{n}\right)  =0.$
For this reason in \cite{kounchevRenderTsachev-BITpaper} the vector
\[
\gamma^{T}:=\left(  L_{\xi}g\left(  t_{2}\right)  ,L_{\xi}g\left(
t_{3}\right)  ,...,L_{\xi}g\left(  t_{n-1}\right)  \right)
\]
is considered. When dealing with clamped $L_{\xi}^{2}-$splines we know
$\mathbf{g}^{T}$ and the boundary conditions $g^{\prime}\left(  t_{1}\right)
=d_{1}$ and $g^{\prime}\left(  t_{n}\right)  =d_{n}.$

In \cite{kounchevRenderTsachev-BITpaper} we proved the following:

\begin{proposition}
\label{Proppsij}Let $t_{1}<....<t_{n}$ and $g$ an $L_{\xi}^{2}-$spline. Then
$g$ restricted to the interval $\left(  t_{j},t_{j+1}\right)  $ for
$j=1,...,n-1$ is equal to
\begin{equation}
\psi_{j}\left(  t\right)  :=L_{\xi}g\left(  t_{j}\right)  A_{j}^{\left[
1\right]  }\left(  t\right)  +L_{\xi}g\left(  t_{j+1}\right)  B_{j}^{\left[
1\right]  }\left(  t\right)  +g\left(  t_{j}\right)  A_{j}^{\left[  2\right]
}\left(  t\right)  +g\left(  t_{j+1}\right)  B_{j}^{\left[  2\right]  }\left(
t\right)  \label{eqpsij}%
\end{equation}
where
\begin{equation}
A_{j}^{\left[  2\right]  }\left(  t\right)  =\frac{\sinh\xi\left(
t-t_{j+1}\right)  }{\sinh\xi\left(  t_{j}-t_{j+1}\right)  }\text{ and }%
B_{j}^{\left[  2\right]  }\left(  t\right)  =\frac{\sinh\xi\left(
t-t_{j}\right)  }{\sinh\xi\left(  t_{j+1}-t_{j}\right)  }, \label{eqdefAB}%
\end{equation}
and
\begin{equation}
A_{j}^{\left[  1\right]  }\left(  t\right)  =\frac{\Phi_{\xi}\left(
t-t_{j+1}\right)  }{\psi_{\xi}\left(  t_{j}-t_{j+1}\right)  }-A_{j}^{\left[
2\right]  }\left(  t\right)  \frac{\Phi_{\xi}\left(  t_{j}-t_{j+1}\right)
}{\psi_{\xi}\left(  t_{j}-t_{j+1}\right)  }, \label{eqA1}%
\end{equation}
and%
\begin{equation}
B_{j}^{\left[  1\right]  }\left(  t\right)  =\frac{\Phi_{\xi}\left(
t-t_{j}\right)  }{\psi_{\xi}\left(  t_{j+1}-t_{j}\right)  }-\frac{\Phi_{\xi
}\left(  t_{j+1}-t_{j}\right)  }{\psi_{\xi}\left(  t_{j+1}-t_{j}\right)
}B_{j}^{\left[  2\right]  }\left(  t\right)  \label{eqB1}%
\end{equation}

\end{proposition}

Here we used the notations
\begin{align*}
\Phi_{\xi}\left(  t\right)   &  =\frac{1}{2\xi^{3}}\left(  \xi t\cosh\left(
\xi t\right)  -\sinh\left(  \xi t\right)  \right) \\
\psi_{\xi}\left(  t\right)   &  =\frac{1}{\xi}\sinh\xi t=L_{\xi}\Phi_{\xi
}\left(  t\right)  .
\end{align*}

Theorem 4.1 in \cite{kounchevRenderTsachev-BITpaper} states:


\begin{theorem}
\label{Thm1} Let $t_{1}<t_{2}<\dots<t_{n}$ be real numbers. Assume that $g$ is
natural $L_{\xi}^{2}-$spline interpolating the data $g_{1},...,g_{n}$. Then
there exists a $n\times\left(  n-2\right)  $ matrix $Q$ and a $\left(
n-2\right)  \times\left(  n-2\right)  $ matrix $R$ such that
\begin{equation}
Q^{T}\mathbf{g}=R\gamma. \label{QTgRgamma}%
\end{equation}
Here $\mathbf{g}$ and $\gamma$ denote the vectors defined in (\ref{g}) and
(\ref{gamma}).

\end{theorem}

The matrices $R$ and $Q$ can be explicitly described:

\begin{theorem}
The matrix $R$ is symmetric and tridiagonal and all entries on the diagonal
and the first off-diagonal are positive. These entries are given by
\begin{align*}
R_{j,j}  &  =\rho\left(  t_{j}-t_{j-1}\right)  +\rho\left(  t_{j+1}
-t_{j}\right)  \text{, }\\
R_{j,j+1}  &  =\sigma\left(  t_{j+1}-t_{j}\right)  .
\end{align*}
where $\rho\left(  t\right)  $ is defined by
\begin{equation}
\rho\left(  t\right)  =\frac{1}{4\xi}\frac{\sinh\left(  2\xi t\right)  -2t\xi
}{\left(  \sinh\xi t\right)  ^{2}}. \label{eqDefrho2}%
\end{equation}
and $\sigma\left(  t\right)  $
\begin{equation}
\sigma\left(  t\right)  =\frac{1}{2\xi}\frac{\xi t\cosh\left(  \xi t\right)
-\sinh\left(  \xi t\right)  }{\sinh^{2}\left(  \xi t\right)  . }
\label{eqDefsig}%
\end{equation}

\end{theorem}

\begin{theorem}
\label{ThmQ} The elements of the $n\times\left(  n-2\right)  $ matrix $Q$ are
given, for $i=1,...,n$ and $j=2,3,...,n-1,$ by the formulas
\begin{equation}
q_{j-1,j}=\frac{\xi}{\sinh\xi\left(  t_{j}-t_{j-1}\right)  },\qquad
q_{j+1,j}=\frac{\xi}{\sinh\xi\left(  t_{j+1}-t_{j}\right)  } \label{Qij}%
\end{equation}
and
\[
q_{jj}=-\xi\left(  \frac{\cosh\xi\left(  t_{j+1}-t_{j}\right)  }{\sinh
\xi\left(  t_{j+1}-t_{j}\right)  }+\frac{\cosh\xi\left(  t_{j}-t_{j-1}\right)
}{\sinh\xi\left(  t_{j}-t_{j-1}\right)  }\right)  .
\]
The rest of the elements in $Q$ are zero.
\end{theorem}

The arguments can be reversed to construct a natural $L_{\xi}^{2}$-spline
interpolating given data $z=\left(  z_{1},...,z_{n}\right)  ^{T}$ as we shall
now outline:

\begin{enumerate}
\item[Step 1] Set $g_{i}=z_{i}$ for $i=1,..n,$

\item[Step 2] Set $x=Q^{T}\mathbf{g}$ and solve $R\gamma=x$.
\end{enumerate}

Then the natural cubic $L_{\xi}^{2}$-spline $g$ interpolating the data $z$ is
then defined on each interval $\left(  t_{j},t_{j+1}\right)  $ by
\[
\psi_{j}\left(  t\right)  :=\gamma_{j}A_{j}^{\left[  1\right]  }\left(
t\right)  +\gamma_{j+1}B_{j}^{\left[  1\right]  }\left(  t\right)  +g_{j}%
A_{j}^{\left[  2\right]  }\left(  t\right)  +g_{j+1}B_{j}^{\left[  2\right]
}\left(  t\right)  .
\]

\section{The main results}

In this section we want to prove analogues of the results in Section $2$ for
clamped $L_{\xi}^{2}-$splines. As outlined in the last section, the $L_{\xi
}^{2}-$spline $g(\cdot)$ coincides on the interval $(t_{j},t_{j+1})$ with the
function
\begin{equation}
\psi_{j}\left(  t\right)  :=L_{\xi}g\left(  t_{j}\right)  A_{j}^{\left[
1\right]  }\left(  t\right)  +L_{\xi}g\left(  t_{j+1}\right)  B_{j}^{\left[
1\right]  }\left(  t\right)  +g\left(  t_{j}\right)  A_{j}^{\left[  2\right]
}\left(  t\right)  +g\left(  t_{j+1}\right)  B_{j}^{\left[  2\right]  }\left(
t\right)  , \label{psi_j}%
\end{equation}
Since in (\ref{psi_j}) the numbers $g(t_{j})=g_{j}$ for $j=1,2,\dots,n$ are
given, generating the spline $g(\cdot)$ reduces to finding the numbers
$L_{\xi}g(t_{j})$ for $j=1,2,\dots,n$. For ease of notation we denote%
\[
\gamma_{j}:=L_{\xi}g(t_{j})\text{ for }j=1,2,\dots,n
\]
and, thus, formula (\ref{psi_j}) transforms to
\begin{equation}
\psi_{j}(t)=\gamma_{j}A_{j}^{[1]}(t)+\gamma_{j+1}B_{j}^{[1]}(t)+g_{j}%
A_{j}^{[2]}(t)+g_{j+1}B_{j}^{[2]}(t). \label{psi_j_2}%
\end{equation}

We next derive a system of $n$ linear algebraic equations in $n$ unknowns, in
which $\{\gamma_{j}\}_{j=1}^{n}$ are the unknowns and $\{g_{j}:=g(t_{j}%
)\}_{j=1}^{n}$ are given. From the continuity conditions at the knots
$t_{2}<t_{3}<\dots<t_{n-1}$ of $g^{\prime}$ we obtain the equalities
\[
g^{\prime}(t_{j}-0)=g^{\prime}(t_{j}+0)\quad\left(  \text{i.e. }\quad
\psi_{j-1}^{\prime}(t_{j})=\psi_{j}^{\prime}(t_{j})\right)
\]
for $j=2,3,\dots,n-1$. From the last and from (\ref{psi_j_2}) we obtain the
following $n-2$ equations:%

\begin{align*}
&  \gamma_{j-1}\frac{d}{dt}A_{j-1}^{[1]}(t_{j})+\gamma_{j}\frac{d}{dt}%
B_{j-1}^{[1]}(t_{j})+g_{j-1}\frac{d}{dt}A_{j-1}^{[2]}(t_{j})+g_{j}\frac{d}%
{dt}B_{j-1}^{[2]}(t_{j})\\
&  =\gamma_{j}\frac{d}{dt}A_{j}^{[1]}(t_{j})+\gamma_{j+1}\frac{d}{dt}%
B_{j}^{[1]}(t_{j})+g_{j}\frac{d}{dt}A_{j}^{[2]}(t_{j})+g_{j+1}\frac{d}%
{dt}B_{j}^{[2]}(t_{j})
\end{align*}
for $j=2,3,\dots,n-1$. These equations may be rewritten in the form
\begin{align}
&  \gamma_{j-1}\frac{d}{dt}A_{j-1}^{[1]}(t_{j})+\gamma_{j}\left(  \frac{d}%
{dt}B_{j-1}^{[1]}(t_{j})-\frac{d}{dt}A_{j}^{[1]}(t_{j})\right)  -\gamma
_{j+1}\frac{d}{dt}B_{j}^{[1]}(t_{j})\nonumber\label{equ_for_gamma}\\
&  =-g_{j-1}\frac{d}{dt}A_{j-1}^{[2]}(t_{j})+g_{j}\left(  \frac{d}{dt}%
A_{j}^{[2]}(t_{j})-\frac{d}{dt}B_{j-1}^{[2]}(t_{j})\right)  +g_{j+1}\frac
{d}{dt}B_{j}^{[2]}(t_{j})\nonumber
\end{align}
for $j=2,\dots,n-1$.

The \emph{clamped boundary} conditions $\psi_{1}^{\prime}(t_{1})=g^{\prime
}(t_{1})=d_{1}$ and $\psi_{n-1}^{\prime}(t_{n})=g^{\prime}(t_{n})=d_{2}$ give
two additional equations:
\begin{equation}
\gamma_{1}\frac{d}{dt}A_{1}^{[1]}(t_{1})+\gamma_{2}\frac{d}{dt}B_{1}%
^{[1]}(t_{1})=-g_{1}\frac{d}{dt}A_{1}^{[2]}(t_{1})-g_{2}\frac{d}{dt}%
B_{1}^{[2]}(t_{1})+d_{1}, \label{equ_clamp_ini}%
\end{equation}
and%
\begin{equation}
\gamma_{n-1}\frac{d}{dt}A_{n-1}^{[1]}(t_{n})+\gamma_{n}\frac{d}{dt}%
B_{n-1}^{[1]}(t_{n})=-g_{n-1}\frac{d}{dt}A_{n-1}^{[2]}(t_{n})-g_{n}\frac
{d}{dt}B_{n-1}^{[2]}(t_{n})+d_{2} \label{equ_clamp_fin}%
\end{equation}
From above three equations we 
obtain a linear system of $n$ equations in $n$ unknowns $\gamma_{1},\gamma
_{2},\dots,\gamma_{n}$ which we want to write in matrix form of the type
\[
\widetilde{R}\mathbf{\widetilde{\mathbf{\gamma}}}^{T}=\widetilde{Q}%
\mathbf{g}^{T}+\mathbf{d}^{T}\text{ for }\mathbf{d}=\underset{n}%
{\underbrace{(d_{1},0,...,0,d_{2}})}%
\]
where $\mathbf{\widetilde{\mathbf{\gamma}}}=\{\gamma_{1},\gamma_{2}%
,\dots,\gamma_{n}\}$ is the vector of the unknowns and $\mathbf{g}%
=\{g_{1},g_{2},\dots,g_{n}\}$ are the given data. For convenience define
as in
\cite{kounchevRenderTsachev-BITpaper}
\[
\widetilde{R}_{j,j-1}=\frac{d}{dt}A_{j-1}^{[1]}(t_{j}),\quad\widetilde
{R}_{j,j}=\frac{d}{dt}B_{j-1}^{[1]}(t_{j})-\frac{d}{dt}A_{j}^{[1]}%
(t_{j}),\quad\widetilde{R}_{j,j+1}=-\frac{d}{dt}B_{j}^{[1]}(t_{j})
\]
for $j=2,\dots,n-1$. The $n\times n$ matrix $\widetilde{R}$ is defined by%
\[
\widetilde{R}=\left(
\begin{array}
[c]{cccccc}%
\frac{d}{dt}A_{1}^{[1]}(t_{1}) & \frac{d}{dt}B_{1}^{[1]}(t_{1}) & 0 & \cdots &
\cdots & 0\\
\widetilde{R}_{2,1} & \widetilde{R}_{2,2} & \widetilde{R}_{2,3} & 0 & \cdots &
0\\
0 & \ddots & \ddots & \ddots & \ddots & \vdots\\
\vdots & \ddots & \ddots & \ddots & \ddots & 0\\
0 & \cdots & 0 & \widetilde{R}_{n-1,n-2} & \widetilde{R}_{n-1,n-1} &
\widetilde{R}_{n-1,n}\\
0 & \cdots & \cdots & 0 & \frac{d}{dt}A_{n-1}^{[1]}(t_{n}) & \frac{d}%
{dt}B_{n-1}^{[1]}(t_{n})
\end{array}
\right)
\]
Obviously, the structure of the matrix $\widetilde{R}$ is tridiagonal. Now we
express the right hand side of the equations by a $n\times n$-matrix
$\widetilde{Q}$ whose entries are defined by
\[
\widetilde{Q}_{j-1,j}=-\frac{d}{dt}A_{j-1}^{\left[  2\right]  }\left(
t_{j}\right)  \text{ and }\widetilde{Q}_{j+1,j}=\frac{d}{dt}B_{j}^{\left[
2\right]  }\left(  t_{j}\right)
\]
and
\[
\widetilde{Q}_{jj}=\frac{d}{dt}A_{j}^{\left[  2\right]  }\left(  t_{j}\right)
-\frac{d}{dt}B_{j-1}^{\left[  2\right]  }\left(  t_{j}\right)  .
\]
Then it is easy to see that $\ $%
\[
\widetilde{Q}=\left(
\begin{array}
[c]{cccccc}%
-\frac{d}{dt}A_{1}^{[2]}(t_{1}) & -\frac{d}{dt}B_{1}^{[2]}(t_{1}) & 0 & 0 &
\cdots & 0\\
\widetilde{Q}_{2,1} & \widetilde{Q}_{2,2} & \widetilde{Q}_{2,3} & 0 &  & 0\\
0 & \ddots & \ddots & \ddots & \ddots & \vdots\\
\vdots & \ddots & \ddots & \ddots & \ddots & 0\\
0 & \cdots & \ddots & \widetilde{Q}_{n-1,n-2} & \widetilde{Q}_{n-1,n-1} &
\widetilde{Q}_{n-1,n}\\
0 & 0 & \cdots & 0 & -\frac{d}{dt}A_{n-1}^{[2]}(t_{n}) & -\frac{d}{dt}%
B_{n-1}^{[2]}(t_{n})
\end{array}
\right)
\]

The following is the main result of this paper:

\begin{theorem}
The matrix $\tilde{R} $ is strictly (row) diagonal dominant, in particular invertible
\end{theorem}

\proof We show that $\widetilde{R}$ is strictly diagonal dominant; recall that
a matrix $A=\{a_{i,j}\}_{i=1j=1}^{n\;\;\;\;n}$ is called strictly (row)
diagonal dominant if $|a_{i,i}|>\sum_{i\neq j}|a_{i,j}|$ for all
$i=1,2,\dots,n$), see e.g. \cite{GreenSilverman}, \cite{Horn-Johnson}. It is
proved in \cite{kounchevRenderTsachev-BITpaper} (see the proof of Theorem
$4.2$) that
\[
\frac{d}{dt}B_{1}^{[1]}(t_{1})=-\sigma(t_{2}-t_{1}),\qquad\text{and }%
\qquad\frac{d}{dt}A_{1}^{[1]}(t_{1})=-\rho(t_{2}-t_{1})
\]
(just before Theorem $4.3$). We also have $\sigma(x)>0$ for $x>0$ and
$\rho(x)>0$ for $x>0$. Since in the proof of Theorem $4.4$ of
\cite{kounchevRenderTsachev-BITpaper} it is established that $\sigma
(x)<\frac{1}{2}\rho(x)$ for $x>0$, we obtain
\[
\left\vert \frac{d}{dt}B_{1}^{[1]}(t_{1})\right\vert =\sigma(t_{2}-t_{1}%
)\leq\frac{1}{2}\rho(t_{2}-t_{1})<\rho(t_{2}-t_{1})=\left\vert \frac{d}%
{dt}A_{1}^{[1]}(t_{1})\right\vert .
\]
Next we deal with the last row of the matrix $R$. Using formula (3.8) from
\cite{kounchevRenderTsachev-BITpaper} we obtain:
\begin{align*}
\frac{d}{dt}A_{n-1}^{[1]}(t_{n})  &  =\frac{\Phi_{\xi}^{\prime}(t_{n}-t_{n}%
)}{\psi_{\xi}(t_{n-1}-t_{n})}-\frac{d}{dt}A_{n-1}^{[2]}(t_{n})\frac{\Phi_{\xi
}(t_{n-1}-t_{n})}{\psi_{\xi}(t_{n-1}-t_{n})}=-\frac{\psi_{\xi}^{\prime}%
(t_{n}-t_{n})}{\psi_{\xi}(t_{n-1}-t_{n})}\frac{\Phi_{\xi}(t_{n-1}-t_{n})}%
{\psi_{\xi}(t_{n-1}-t_{n})}\\
&  =-\frac{\Phi_{\xi}(t_{n-1}-t_{n})}{\psi_{\xi}^{2}(t_{n-1}-t_{n})}%
=-\sigma(t_{n-1}-t_{n})=\sigma(t_{n}-t_{n-1})
\end{align*}
(the second equality above follows from the proof of Theorem $4.2$ of
\cite{kounchevRenderTsachev-BITpaper}). Also, we have the equality $\frac
{d}{dt}B_{n-1}^{[1]}(t_{n})=\rho(t_{n}-t_{n-1})$ (which follows from
\cite{kounchevRenderTsachev-BITpaper} -- formula (3.3) and the proof of
Theorem $4.2$). Hence, again
\[
\frac{d}{dt}A_{n-1}^{[1]}(t_{n})=\sigma(t_{n}-t_{n-1})\leq\frac{1}{2}%
\rho(t_{n}-t_{n-1})<\rho(t_{n}-t_{n-1})=\frac{d}{dt}B_{n-1}^{[1]}(t_{n}).
\]
So, we established the diagonal dominance for the first and for the last rows.
The respective dominance for the rows of $R$ with numbers $2,3,\dots,n-1$ is
established in \cite{kounchevRenderTsachev-BITpaper}.

The strict (row) diagonal dominance implies that the matrix $R$ is invertible,
by Theorem $6.1.10$ in \cite{Horn-Johnson}.

\endproof

Thus we obtain
\[
\widetilde{\gamma}=\left(  \widetilde{R}\right)  ^{-1}\left(  \widetilde
{Q}\mathbf{g}^{T}+\mathbf{d}^{T}\right)
\]
and we may proceed as at the end of Section 2 for computing the $L-$spline
$g\left(  t\right)  $.

The present $L-$splines may be used to create multivariate clamped polysplines, cf. \cite{KounchevBOOK}. 
The last may be used as an alternative of the PINNs for solving equations of Mathematical 
Physics.  

\section{Figure}

The fast algorithm above was implemented in Matlab. As an example we
provide the figure of a clamped $L-$spline for $L=L_{\xi}^{2}$ with $\xi=5,$
on the interval $\left[  0,1\right]  $ with knots%
\[
t_{1}=0,t_{2}=0.1667,t_{3}=0.3333,t_{4}=0.5000,t_{5}=0.6667,t_{6}%
=0.8333,t_{7}=1
\]
with derivatives $g^{\prime}\left(  t_{1}\right)  =25$ and $g^{\prime}\left(
t_{7}\right)  =25$ interpolating the function $f\left(  t\right)  =\sin\left(
25t\right)  .$ In the figure, the interpolating $L_{5}^{2}-$spline is compared
with the linear spline interpolating $f\left(  t\right)  .$

\begin{figure}[htb]
	\centering
	\includegraphics[width=1.0\linewidth,keepaspectratio]{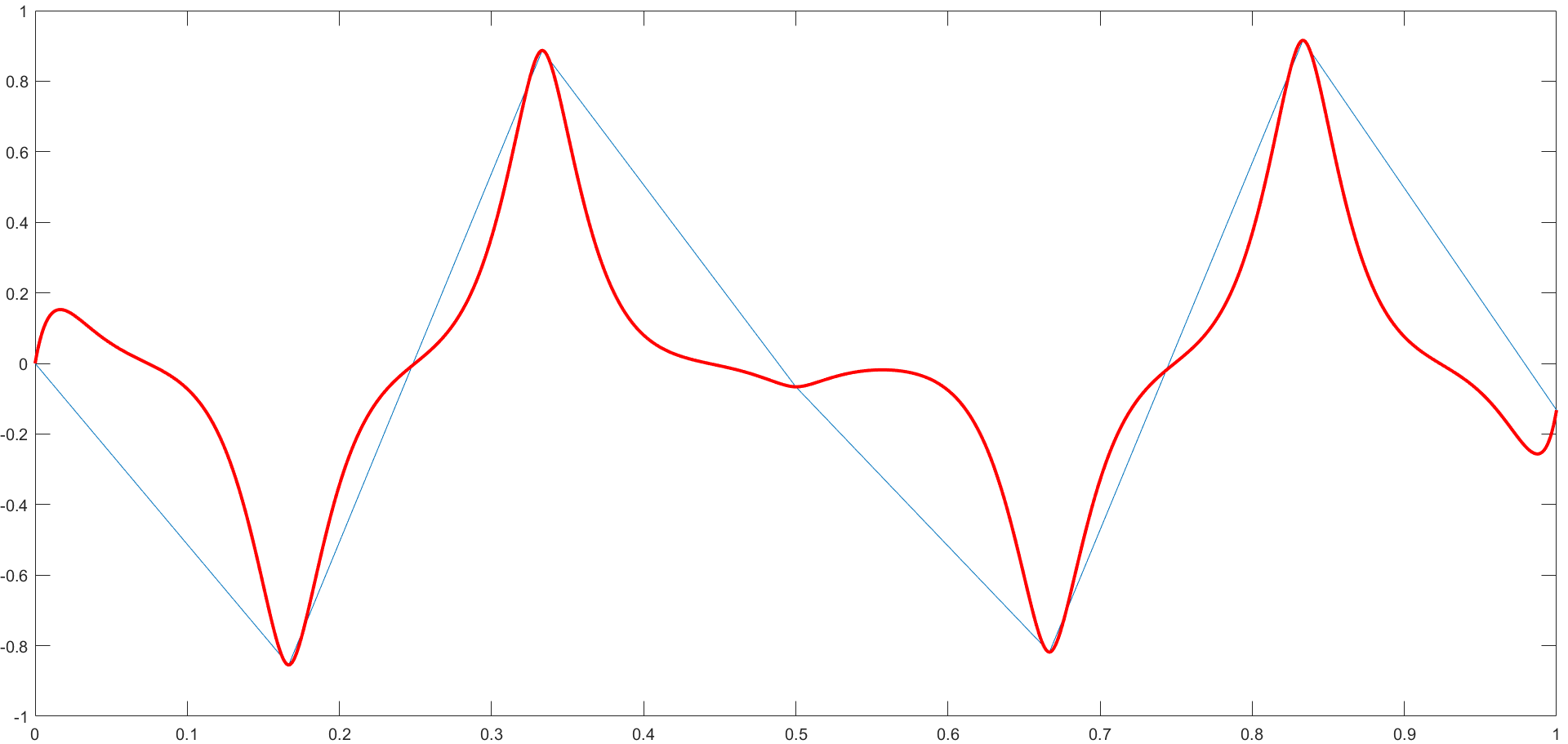}
	\caption{Clamped L-spline}
\end{figure}

ACKNOWLEDGEMENT

The work of OK was partially supported by a project with Bulgarian NSF within the Multilateral competition for scientific and technological collaboration in the Danube region - 2024: Application of novel AI methods in analyzing Big data in Astrophysics and Physics: multidisciplinary and multilateral effort, entry number BG-175467353-2024-18-0018 / FNI-238 of  17.01.2025. The work of GS and HR was partially supported by project with Bulgarian NSF, KP-06-RILA. 
The work of TT was partially supported by the Centre of Excellence in
Informatics and ICT under the Grant No BG16RFPR002-1.014-0018-C01, financed by
the Research, Innovation and Digitalization for Smart Transformation Programme
2021-2027 and co-financed by the European Union.

\bigskip
\textbf{Affiliations: }

\medskip
CONTACT PERSON: 
\medskip
O. Kounchev, Institute of Mathematics and Informatics - BAS, Centre of
Excellence in Informatics and Information and Communication Technologies,
Sofia, Bulgaria; email: kounchev@math.bas.bg

\medskip
H. Render, Institute of Mathematics and Informatics - BAS; email: render65@gmx.com

\medskip
G. Simeonov, Institute of Mathematics and Informatics - BAS; email: gsimeonov@math.bas.bg

\medskip
Ts. Tsachev, Institute of Mathematics and Informatics - BAS, Centre of
Excellence in Informatics and Information and Communication Technologies,
Sofia, Bulgaria;  tsachev@math.bas.bg

\bigskip\

\end{document}